\documentclass[12pt]{amsart}
\usepackage{amsmath,amssymb,fullpage}
\usepackage[active]{srcltx} 

\newtheorem{theorem}{Theorem}[section]

\newtheorem{corollary}[theorem]{Corollary}

\newtheorem{remark}[theorem]{Remark}

\newtheorem{example}[theorem]{Example}

\numberwithin{equation}{section}

\newcommand{\R}{\mathbb{R}}
\newcommand{\N}{\mathbb{N}}
\newcommand{\Z}{\mathbb{Z}}
\newcommand{\C}{\mathbb{C}}

\newcommand{\cF}{\mathcal{F}}
\newcommand{\cM}{\mathcal{M}}
\newcommand{\dis}{\displaystyle}
\newcommand{\supp}{\textup{supp}}
\newcommand{\dist}{\textup{dist}}
\newcommand{\diam}{\textup{diam}}
\newcommand{\ca}{\textup{cap}}
\newcommand{\eps}{\varepsilon}
\newcommand{\li}{\textup{li}}

\begin{document}

\title[Equidistribution of points via energy]
{Equidistribution of points via energy}%
\author{Igor E. Pritsker}%

\thanks{Research was partially supported by the National Security
Agency, and by the Alexander von Humboldt Foundation.}

\address{Department of Mathematics, Oklahoma State University, Stillwater, OK 74078, U.S.A.}%
\email{igor@math.okstate.edu}




\begin{abstract}

We study the asymptotic equidistribution of points with discrete energy close to Robin's constant of a compact set in the plane. Our main tools are the energy estimates from potential theory. We also consider the quantitative aspects of this equidistribution. Applications include estimates of growth for the Fekete and Leja polynomials associated with large classes of compact sets, convergence rates of the discrete energy approximations to Robin's constant, and problems on the means of zeros of polynomials with integer coefficients.

\end{abstract}

\maketitle


\section{Asymptotic equidistribution of discrete sets}

Let $E$ be a compact set in the complex plane $\C.$ Given a set of points $Z_n=\{z_{k,n}\}_{k=1}^n\subset\C,\ n\ge 2$, the associated Vandermonde determinant is
\[
V(Z_n):=\prod_{1\le j<k\le n} (z_{j,n}-z_{k,n}).
\]
Let the $n$-th diameter of $E$ be defined by
\[
\delta_n(E):=\max_{Z_n\subset E} |V(Z_n)|^{\frac{2}{n(n-1)}}.
\]
A set of points $\cF_n$ is called the $n$-th Fekete points of $E$ if it achieves the above maximum. The classical result of Fekete \cite{Fe} states that $\delta_n(E),\ n\ge 2,$ form a decreasing sequence that converges to a limit called the transfinite diameter $\delta(E)$. Szeg\H{o} \cite{Sz2} found that $\delta(E)$ is equal to the logarithmic capacity cap$(E)$ from potential theory, which is defined as follows. For a Borel measure $\mu$ with compact support, define its energy by \cite[p. 54]{Ts}
\[
I[\mu]:= \iint \log \dis\frac{1}{|z-t|} \, d \mu(t)d \mu(z).
\]
Consider the problem of finding the minimum energy
\[
V_E:=\inf_{\mu\in\cM(E)} I[\mu],
\]
where $\cM(E)$ is the space of all positive unit Borel measures supported on $E$. The capacity of $E$ is given by
\[
\ca(E):=e^{-V_E}.
\]
If Robin's constant $V_E$ is finite (i.e. $\ca(E)\neq 0$), then the infimum is attained by the equilibrium measure $\mu_E\in\cM(E)$ \cite[p. 55]{Ts}, which is a unique probability measure expressing the steady state distribution of charge on the conductor $E$. For detailed expositions of potential theory, we refer the reader to the books of Ransford \cite{Ra}, Tsuji \cite{Ts}, and Landkof \cite{La}.

Consider the counting measure $\tau(Z_n)$ for the set $Z_n=\{z_{k,n}\}_{k=1}^n$, given by
\[
\tau(Z_n) := \frac{1}{n} \sum_{k=1}^n \delta_{z_{k,n}},
\]
where $\delta_{z_{k,n}}$ is the unit point mass at $z_{k,n}\in Z_n$.
It is clear that $I[\tau(Z_n)]=\infty,$ but we can define the {\em discrete} energy of $\tau(Z_n)$ (or of the set $Z_n$) by setting
\[
\hat{I}[\tau(Z_n)] := - \log |V(Z_n)|^{\frac{2}{n(n-1)}} = \frac{2}{n(n-1)} \sum_{1\le j<k\le n} \log\frac{1}{|z_{j,n}-z_{k,n}|}.
\]
Note that the discrete energy $\hat{I}[\tau(Z_n)]$ is finite if and only if all points of $Z_n$ are distinct. The Fekete-Szeg\H{o} results may be restated as
\[
\lim_{n\to\infty} \inf_{Z_n\subset E} \hat{I}[\tau(Z_n)] = \lim_{n\to\infty} \hat{I}[\tau(\cF_n)] = \lim_{n\to\infty} \left(-\log \delta_n(E)\right) = V_E=I[\mu_E],
\]
which simply indicates that the discrete approximations of the minimum energy converge to Robin's constant, see \cite[p. 153]{Ra}. It is also well known that the counting measures $\tau(\cF_n)$ converge to $\mu_E$  in the weak-* topology (written $\tau(\cF_n) \stackrel{*}{\rightarrow} \mu_E$) as $n\to\infty$, provided that $\ca(E)>0$ \cite[p. 226]{AB}. Such equidistribution property is shared by many sequences of discrete sets whose energies converge to Robin's constant, see Andrievskii and Blatt \cite{AB} for history and references. Our new equidistribution result is as follows.

For an arbitrary compact set $E\subset\C,$ let $\Omega_E$ be the unbounded connected component of $\overline{\C}\setminus E$. If $\ca(E)>0$ then the Green function $g_E(z,\infty)$ for $\Omega_E$ with pole at $\infty$ \cite[p. 14]{Ts} is well defined. We use the quantity
\[
m_E(Z_n) := \frac{1}{n} \sum_{z_{k,n}\in\Omega_E} g_E(z_{k,n},\infty)
\]
to measure how close $Z_n$ is to $E$. If $Z_n\cap\Omega_E=\emptyset$ then we set $m_E(Z_n)=0$ by definition.

\begin{theorem} \label{thm1.1}
Let $E\subset\C$ be compact, $\ca(E)>0$. If the sets $Z_n=\{z_{k,n}\}_{k=1}^n\subset\C,\ n\ge 2,$ satisfy
\begin{align} \label{1.1}
\lim_{n\to\infty} \hat{I}[\tau(Z_n)] = V_E
\end{align}
and
\begin{align} \label{1.2}
\lim_{n\to\infty} m_E(Z_n) = 0,
\end{align}
then
\begin{align} \label{1.3}
\left\{
\begin{array}{l}
(i)\ \dis\tau(Z_n) \stackrel{*}{\rightarrow} \mu_E \mbox{ as }n\to\infty, \\ \\ (ii) \dis\lim_{R\to\infty} \lim_{n\to\infty} \frac{1}{n} \sum_{|z_{k,n}|\ge R} \log|z_{k,n}| = 0.
\end{array}
\right.
\end{align}
Conversely, \eqref{1.2} holds for any sequence of the sets $Z_n=\{z_{k,n}\}_{k=1}^n\subset\C,\ n\in\N,$ satisfying \eqref{1.3}.
\end{theorem}

When $Z_n\subset E,$ we clearly have that $m_E(Z_n)=0$ for all $n\ge 2,$ and \eqref{1.1} implies the well known fact that $\tau(Z_n) \stackrel{*}{\rightarrow} \mu_E \mbox{ as }n\to\infty.$ A new feature of the above result is that $Z_n$ is not required to be a subset of $E$. Introduction of $m_E(Z_n)$ is inspired by the generalized Mahler measure that was used in \cite{Pr4} to study the asymptotic zero distribution for polynomials with integer coefficients. Theorem \ref{thm1.1} is a generalization of Theorem 2.1 from \cite{Pr4}. The majority of equidistribution results in analysis are stated in terms of zeros of polynomials, with the assumptions expressed via the supremum norms $\|P_n\|_E:=\sup_{z\in E} |P_n(z)|$ of polynomials, see \cite{AB}. We recall one of the most frequently used results of this kind, due to Blatt, Saff and Simkani \cite{BSS}.

\noindent
{\bf Theorem BSS.} {\em Let $E\subset\C$ be a compact set, $\ca(E)>0$, and set $E^*:=\supp(\mu_E)$.  If the sets $Z_n=\{z_{k,n}\}_{k=1}^n\subset\C,\ n\ge 2,$ and the corresponding polynomials $P_n(z):=\prod_{k=1}^n (z-z_{k,n})$ satisfy
\begin{align} \label{1.4}
\lim_{n\to\infty} \|P_n\|_{E^*}^{1/n} = \ca(E)
\end{align}
and
\begin{align} \label{1.5}
\lim_{n\to\infty} \tau_n(A) = 0,
\end{align}
where $\tau_n:=\tau(Z_n)$ and $A$ is any closed set in the bounded components of $\C\setminus E^*$, then
\begin{align} \label{1.6}
\tau_n \stackrel{*}{\rightarrow} \mu_E \mbox{ as }n\to\infty.
\end{align}
}
We note that \eqref{1.4} implies \eqref{1.2} because
\begin{align*}
m_E(Z_n) &\le \frac{1}{n} \sum_{k=1}^n g_E(z_{k,n},\infty) = \int g_E(z,\infty)\,d\tau_n(z) \\ &= \int \left(\int\log|z-t|\,d\mu_E(t) - \log \ca(E) \right)\,d\tau_n(z) \\ &= \int\int\log|z-t|\,d\tau_n(z)\,d\mu_E(t) - \log \ca(E) \\ &= \int\log|P_n(t)|^{1/n}\,d\mu_E(t) - \log \ca(E) \le \log\|P_n\|_{E^*}^{1/n} - \log \ca(E),
\end{align*}
where we used the standard representation of $g_E(z,\infty)$ given in \eqref{5.2}. Thus using $m_E(Z_n)$ instead of $\|P_n\|_{E^*}$ (or $\|P_n\|_E$) gives stronger results, in general. However, conditions \eqref{1.5} and \eqref{1.1} are substantially different, so that Theorem \ref{thm1.1} and Theorem BSS complement each other.

The following fact about the supremum norms of polynomials is of independent interest.

\begin{theorem} \label{thm1.2}
Let $E\subset\C$ be a regular compact set. Suppose that the sets $Z_n=\{z_{k,n}\}_{k=1}^n\subset\C,\ n\ge 2,$ satisfy \eqref{1.1}.
We have that
\begin{align} \label{1.7}
\lim_{n\to\infty} \|P_n\|_E^{1/n} = \ca(E)
\end{align}
for the polynomials $P_n(z)=\prod_{k=1}^n (z-z_{k,n})$ is equivalent to
\eqref{1.2} or \eqref{1.3}.
\end{theorem}
Regularity is understood here in the sense of the Dirichlet problem for $\Omega_E$, which means that the limiting boundary values of $g_E(z,\infty)$ in $\Omega_E$ are all zero, see \cite[p. 82]{Ts}. Regularity of $E$ also implies that $\ca(E)>0$. We recall that any monic polynomial $P_n$ of degree $n$ satisfies $\|P_n\|_E \ge (\ca(E))^n,$ see \cite[p. 16]{AB}. Thus \eqref{1.7} (and \eqref{1.4}) means that $P_n$ have asymptotically minimal supremum norms on $E$.

Our results have clear analogues in $\R^n,\ n>2,$ where one should use the theory of Newtonian potentials and the associated Green functions. Such extensions are also valid for the majority of quantitative estimates stated in the next section.

\section{Rate of convergence and discrepancy in equidistribution}

This section is devoted to the quantitative estimates of how close $\tau(Z_n)$ is to the equilibrium measure $\mu_E$, which are often called discrepancy estimates. Consider a class of continuous test functions $\phi:\R^2\to\R$ with compact support in the plane $\R^2=\C.$ Recall that $\tau(Z_n) \stackrel{*}{\rightarrow} \mu_E \mbox{ as }n\to\infty$ means
\[
\lim_{n\to\infty} \frac{1}{n} \sum_{k=1}^n \phi(z_{k,n}) = \lim_{n\to\infty} \int\phi\,d\tau(Z_n) = \int\phi\,d\mu_E.
\]
Let
\[
\omega_{\phi}(r):=\sup_{|z-\zeta|\le r} |\phi(z)-\phi(\zeta)|
\]
be the modulus of continuity of $\phi$ in $\C$. We also require that the functions $\phi$ have finite Dirichlet integral
\[
D[\phi]:= \iint_{\R^2} \left(\phi_x^2+\phi_y^2\right)\,dxdy,
\]
where it is assumed that the partial derivatives $\phi_x$ and $\phi_y$ exist a.e. on $\R^2$ in the sense of the area measure. Define the distance from a point $z\in\C$ to a compact set $E$ by
\[
d_E(z):=\min_{t\in E} |z-t|.
\]

\begin{theorem} \label{thm2.1}
Let $E\subset\C$ be an arbitrary compact set of positive capacity, and let $\phi:\C\to\R$ be a continuous function with compact support such that $D[\phi]<\infty.$ If $Z_n=\{z_{k,n}\}_{k=1}^n\subset\C,\ n\ge 2,$ then we have for any $r>0$ that
\begin{align} \label{2.1}
\left|\frac{1}{n} \sum_{k=1}^n \phi(z_{k,n}) - \int\phi\,d\mu_E\right| \le \omega_{\phi}(r) + \sqrt{\frac{D[\phi]}{2\pi}}\, \sqrt{I},
\end{align}
where
\begin{align} \label{2.2}
I = 2m_E(Z_n) + \frac{n-1}{n} \hat{I}[\tau(Z_n)] - V_E - \frac{\log{r}}{n} + 2\max_{d_E(z)\le 2r} g_E(z,\infty).
\end{align}
\end{theorem}

The classical discrepancy results for the unit circle and the segment $[-1,1]$ are due to Erd\H{o}s and Tur\'an, cf. \cite{ET} and \cite{ET2}. They were developed by Ganelius \cite{Ga}, Amoroso and Mignotte \cite{AM}, and many others, see \cite{AB} for more history and references. Further generalizations and improvements are due to Blatt \cite{Bl}, Totik \cite{To2}, Blatt and Mhaskar \cite{BM}, Andrievskii and Blatt \cite{AB2}-\cite{AB3}, and others.  The ideas of applying energy estimates originated in part in the work of Kleiner \cite{Kl}, and were subsequently used by Sj\"{o}gren \cite{Sj1}-\cite{Sj2} and Huesing \cite{Hus}, see \cite[Ch. 5]{AB}. Favre and Rivera-Letelier \cite{FR} proved a result for the unit circle in number theoretic terms, related to Theorem \ref{thm2.1}. A predecessor of Theorem \ref{thm2.1} may be found in Theorem 4.3 of \cite{Pr4}, where we studied the asymptotic distribution of algebraic numbers, and answered some questions of Schur \cite{Sch}. Perhaps the most interesting new feature of Theorem \ref{thm2.1} is its generality. All previous discrepancy results imposed strict geometric conditions on the set $E$. A typical application of our result is given by a sequence of sets $Z_n$ satisfying \eqref{1.1} and \eqref{1.2}. We choose $r=r_n\to 0$ as $n\to\infty$, so that the right hand side of \eqref{2.1} tends to 0 with a certain rate under the mere assumption that the Green function $g_E(z,\infty)$ is continuous at the boundary points of $\Omega_E$ (i.e. $E$ is regular). For the effective estimates, one would usually take $r_n=c/n^a,$ with $a,c>0,$ and consider sets with uniformly H\"older continuous Green functions. We state the condition of H\"older continuity for $g_E(z,\infty)$ in the following form:
\begin{align} \label{2.3}
g_E(z,\infty) \le C(E) (d_E(z))^s, \quad z\in\Omega_E,
\end{align}
where $C(E)>0$ and $0<s\le 1.$ It holds for all uniformly perfect sets, various Cantor-type sets, and many other compact sets, see Carleson and Totik \cite[pp. 562--563]{CT} and Totik \cite{To} for the discussion and further references. Uniformly perfect sets form the widest known class given by a natural geometric condition, for which \eqref{2.3} is valid. A compact set $E$ is called uniformly perfect if there exist constants $c,d>0$ such that for any $z\in E$ and any $r\in(0,d)$ there is $\zeta\in E$ satisfying $cr<|z-\zeta|<r.$ Several interesting characterizations and many applications of uniformly perfect sets are discussed in the survey by Sugawa \cite{Su}, where the reader may also find history and numerous additional references. Uniformly perfect sets trivially include compact sets consisting of finitely many non-degenerate connected components.

We consider an application to the ``near-Fekete" points, i.e., to the sets $Z_n\subset E$ whose discrete energies are close to $V_E$.

\begin{theorem} \label{thm2.2}
Let $E\subset\C$ be a compact set, $\ca(E)>0,$ such that the H\"older condition \eqref{2.3} holds for $g_E(z,\infty)$. Suppose that  $Z_n=\{z_{k,n}\}_{k=1}^n\subset E,\ n\ge 2,$ satisfy
\begin{align} \label{2.4}
\hat{I}[\tau(Z_n)] - V_E \le C_1\, \frac{\log{n}}{n}, \quad n\ge 2,
\end{align}
where $C_1>0$ is independent of $Z_n$, and consider $\Omega_n:=\{z\in\Omega_E: g_E(z,\infty)>1/n\}$. Then we have for the polynomials $P_n(z):=\prod_{k=1}^n (z-z_{k,n})$ that
\begin{align} \label{2.5}
\left| \frac{1}{n}\log|P_n(z)|+ V_E - g_E(z,\infty) \right| \le C_2\,\frac{\log{n}}{\sqrt{n}}, \quad z\in\overline\Omega_n,\ n\ge 2,
\end{align}
where $C_2>0$ is independent of $z$ and $Z_n.$ Furthermore,
\begin{align} \label{2.6}
\log \|P_n\|_E + n V_E \le C_2\,\sqrt{n}\log{n} + 1, \quad n\ge 2,
\end{align}
and
\begin{align} \label{2.7}
- C_3 \, \frac{\log{n}}{\sqrt{n}} \le \hat{I}[\tau(Z_n)] - V_E, \quad n\ge 2,
\end{align}
with $C_3>0$ being independent of $Z_n$.
\end{theorem}

One of the standard applications for arrays of equidistributed points is to interpolation of analytic functions. Thus \eqref{2.5} and \eqref{2.6} are used in the effective estimates of convergence rates for Lagrange interpolation polynomials via Hermite interpolation formula, see Gaier \cite[Ch. 2]{Gai} and Walsh \cite[Ch. 4]{Wa}. The Fekete points $\cF_n=\{\zeta_{k,n}\}_{k=1}^n$ represent very convenient nodes for interpolation, and one can easily verify that Theorem \ref{thm2.2} applies in this case. However, they are difficult to find explicitly and even numerically, as all points of $\cF_n$ change with $n$. Another choice of interpolation nodes frequently used in practice is given by Leja points. They hold advantage of being defined as a sequence. If $E\subset\C$ is a compact set of positive capacity, then the Leja points $\{\xi_k\}_{k=0}^{\infty}$ are defined recursively in the following way.  We choose $\xi_0\in E$ as an arbitrary point. When $\{\xi_k\}_{k=0}^n$ are selected, we choose the next point $\xi_{n+1}\in E$ as a point satisfying
\[
\prod_{k=0}^{n} |\xi_{n+1}-\xi_k| = \max_{z\in E} \prod_{k=0}^{n} |z-\xi_k|.
\]
It is known that Leja points are equidistributed in $E$, cf. \cite{BBCL}. However, the properties of Leja points are not studied as well as those of Fekete points. Thus Theorem \ref{thm2.2} provides new information about Leja points and corresponding polynomials for quite general sets.

\begin{corollary} \label{cor2.3}
Theorem \ref{thm2.2} holds for Fekete and Leja points.
\end{corollary}

Surveys of results on Fekete points and Fekete polynomials may be found in Korevaar \cite{Ko}, Andrievskii and Blatt \cite{AB} and Korevaar and Monterie \cite{KM}. We note that the estimates of Theorem \ref{thm2.2} can be improved for the Fekete points and Fekete polynomials of a set $E$ satisfying more restrictive smoothness conditions. Results on Leja points and interpolation may be found in Bloom, Bos, Christensen and Levenberg \cite{BBCL}, while G\"otz \cite{Go} considered questions of discrepancy in their distribution.

We now state a consequence of Theorem \ref{thm2.1} for the Lipschitz continuous functions $\phi.$
\begin{theorem} \label{thm2.4}
Let $E\subset\C$ be a compact set, $\ca(E)>0,$ with Green function satisfying the H\"older condition \eqref{2.3}. Suppose that $\phi:\C\to\R$ is a Lipschitz continuous function with compact support. If $Z_n=\{z_{k,n}\}_{k=1}^n\subset\C,\ n\ge 2,$ satisfy \eqref{2.4}, then
\begin{align} \label{2.8}
\left|\frac{1}{n}\sum_{k=1}^n \phi(z_{k,n}) - \int\phi\,d\mu_E\right| \le C_4 \sqrt{\max\left(\frac{\log{n}}{n},m_E(Z_n)\right)}, \quad n\ge 2,
\end{align}
where $C_4>0$ does not depend on $Z_n$.
\end{theorem}

As an application, we give an estimate of how close are the complex moments of the discrete measures $\tau(Z_n)$ to the moments of $\mu_E.$ Similar result for the real moments $\int |z|^m \,d\mu(z)$ may also be easily deduced from Theorem \ref{thm2.4}.

\begin{corollary} \label{cor2.5}
Let $E\subset\C$ be a compact set, $\ca(E)>0,$ with Green function satisfying \eqref{2.3}. If $Z_n=\{z_{k,n}\}_{k=1}^n\subset E,\ n\ge 2,$ satisfy \eqref{2.4}, then for each $m\in\N$ we have
\begin{align} \label{2.9}
\left|\frac{1}{n}\sum_{k=1}^n z_{k,n}^m - \int z^m\,d\mu_E(z)\right| \le C_5 \sqrt{\frac{\log{n}}{n}}, \quad n\ge 2,
\end{align}
where $C_5>0$ does not depend on $Z_n$ (but depends on $m$ and $E$).
\end{corollary}

Several applications of this kind to Schur's problems on means of algebraic numbers \cite{Sch} were given in \cite{Pr3} and \cite{Pr4}. We would like to highlight an interesting fact that Schur's paper \cite{Sch} prompted Fekete to introduce his transfinite diameter in \cite{Fe}. While the work of Fekete \cite{Fe} is well known in analysis, and is clearly considered a cornerstone of the area dealt with in this paper, the fundamental nature of Schur's work \cite{Sch} has become somewhat obscured with time. In fact, Schur's ideas contained in \cite{Sch} started several important areas of research in analysis and number theory.

We state below a generalization of Theorems 3.1 and 3.4 from \cite{Pr4} to much more general sets.

\begin{theorem} \label{thm2.6}
Let $E\subset\C$ be a compact set, $\ca(E)=1,$ with Green function satisfying \eqref{2.3}. Suppose that $\phi:\C\to\R$ is a Lipschitz continuous function with compact support. If $P_n(z) = a_n\prod_{k=1}^n (z-z_{k,n}),\ a_n\neq 0,$ is a polynomial with integer coefficients and simple zeros $Z_n=\{z_{k,n}\}_{k=1}^n\subset\C,\ n\ge 2,$ then
\begin{align} \label{2.10}
\left|\frac{1}{n}\sum_{k=1}^n \phi(z_{k,n}) - \int\phi\,d\mu_E\right| \le C_6 \sqrt{\max\left(\frac{\log(n|a_n|)}{n},m_E(Z_n)\right)}, \quad n\ge 2,
\end{align}
where $C_6>0$ does not depend on $P_n$.
\end{theorem}

Let $\Z_n^s(E,M)$ be a class of polynomials $P_n(z)=a_nz^n+\ldots$ with integer coefficients and simple zeros in a set $E\subset\C$, satisfying $0<|a_n|\le M$ for a fixed number $M>0.$ Schur \cite[\S 4-8]{Sch} studied the limit behavior of the arithmetic means $A_n$ of zeros for polynomials from $\Z_n^s(E,M)$ as $n\to\infty.$ For $E=D$ the closed unit disk, Schur proved that
\begin{align*}
\limsup_{n\to\infty} |A_n| \le 1-\sqrt{e}/2 < 0.1757.
\end{align*}
We showed \cite{Pr3} that $\lim_{n\to\infty} A_n = 0$ for any sequence of polynomials from Schur's classes $\Z_n^s(D,M),\ n\in\N,$ as a consequence of the asymptotic equidistribution of zeros near the unit circle. We also gave estimates of the convergence rates for $A_n.$ The following result generalizes Corollary 1.6 from \cite{Pr3}, as well as Corollary 3.5 from \cite{Pr4}, to centrally symmetric compact sets of capacity 1.

\begin{corollary} \label{cor2.7}
Let $E\subset\C$ be a compact set, $\ca(E)=1,$ symmetric with respect to the origin, with $g_E(z,\infty)$ satisfying \eqref{2.3}. If $P_n(z) = a_n\prod_{k=1}^n (z-z_{k,n})\in\Z_n^s(E,M)$ then
\begin{align} \label{2.11}
\left|\frac{1}{n}\sum_{k=1}^n z_{k,n}\right| \le C_7 \sqrt{\frac{\log{n}}{n}},\quad n\ge \max(M,2).
\end{align}
where $C_7>0$ does not depend on $P_n$.
\end{corollary}

Note that \eqref{2.5}, \eqref{2.6}, \eqref{2.9} and \eqref{2.11} are sharp up to certain logarithmic factors, even for polynomials with integer coefficients and $E=D$ the closed unit disk.

\begin{example} \label{ex2.8}
Let $p_m$ be the $m$th prime number in the increasing ordering of primes. Define the monic polynomials
\[
P_n(z):=\prod_{m=1}^{k} \frac{z^{p_m}-1}{z-1}, \quad k\in\N,
\]
and note that each $P_n$ has simple zeros $Z_n=\{z_{j,n}\}_{j=1}^n$ at the roots of unity, and integer coefficients. Hence the discriminant $\Delta(P_n) =  (V(Z_n))^2$ is a non-zero integer, see \cite[p. 24]{Pra}. We conclude that $|\Delta(P_n)|\ge 1$ and $\hat{I}[\tau(Z_n)] = - \log |\Delta(P_n)|^{\frac{1}{n(n-1)}} \le 0 = V_D,$ so that \eqref{2.4} is satisfied. Using number theoretic arguments, we show in the proof that the degree of $P_n$ is
\begin{align*}
 n = \sum_{m=1}^k p_m - k = \frac{k^2 \log k}{2} + o(k^2 \log k)\quad \mbox{as } k\to\infty,
\end{align*}
and that
\begin{align*}
\|P_n\|_D = P_n(1) = \prod_{m=1}^k p_m \ge e^{c_1 \sqrt{n\log n}},\quad n\ge 2,
\end{align*}
with a constant $c_1>0.$ Therefore, the upper bound in \eqref{2.6} is of correct order of magnitude up to the factor $\sqrt{\log n}$. The same conclusion is true for \eqref{2.5} by the Maximum Principle. Furthermore, since the sum of roots of each $(z^{p_m}-1)/(z-1)$ is equal to $-1,$ we obtain for the roots of $P_n$ that
\[
\left|\frac{1}{n}\sum_{j=1}^n z_{j,n}\right| = \frac{k}{n} \ge \frac{c_2}{\sqrt{n\log{n}}},
\]
where $c_2>0.$ Hence \eqref{2.9} and \eqref{2.11} are sharp up to the factor $\log{n}.$
\end{example}

In conclusion, we mention that Schur \cite{Sch} also considered the limit behavior of the arithmetic means of zeros for polynomials from $\Z_n^s(E,M)$ when $E=(0,\infty)$ and $E=\R.$ The case $E=(0,\infty)$ was developed by Siegel \cite{Si} and others, see \cite{Pr4} for history, more references, and new results.

\section{Proofs}

We give a brief review of basic facts from potential theory. A complete account may be found in the books by Ransford \cite{Ra}, Tsuji \cite{Ts}, and Landkof \cite{La}. For a Borel measure $\mu$ with compact support, define its potential \cite[p. 53]{Ts} by
\[
U^{\mu}(z):=\int \log\frac{1}{|z-t|}\,d\mu(t), \quad z\in\C.
\]
It is known that $U^{\mu}(z)$ is a superharmonic function in $\C$, which is harmonic outside $\supp(\mu)$. If $U^{\mu_E}(z)$ is the equilibrium (conductor) potential for $E$, then Frostman's theorem \cite[p. 60]{Ts} gives that
\begin{align} \label{5.1}
U^{\mu_E}(z) \le V_E,\ z\in\C, \quad \mbox{and} \quad U^{\mu_E}(z) = V_E \mbox{ q.e. on }E.
\end{align}
The second statement means that equality holds quasi everywhere on $E$, i.e., except for a subset of zero capacity in $E$. This may be made even more precise, as $U^{\mu_E}(z) = V_E$ for any $z\in\overline\C\setminus \overline\Omega_E.$ Hence $U^{\mu_E}(z) = V_E$ for any $z$ in the interior of $E$ \cite[p. 61]{Ts}. Furthermore, $U^{\mu_E}(z) = V_E$ for $z\in\partial\Omega_E$ if and only if $z$ is a regular point for the Dirichlet problem in $\Omega_E$ \cite[p. 82]{Ts}. We mention a well known connection of the equilibrium potential for $E$ with the Green function $g_E(z,\infty)$ for $\Omega_E$ with pole at $\infty$:
\begin{align} \label{5.2}
g_E(z,\infty) = V_E - U^{\mu_E}(z),\quad z\in\C.
\end{align}
This gives a standard extension of $g_E(z,\infty)$ from $\Omega_E$ to the whole plane $\C,$ see \cite[p. 82]{Ts}. Thus $g_E(z,\infty)=0$ for quasi every $z\in\partial\Omega_E$, and $g_E(z,\infty)=0$ for any $z\in\overline\C\setminus\overline\Omega_E,$ by \eqref{5.1} and \eqref{5.2}.

\begin{proof}[Proof of Theorem \ref{thm1.1}]
Set $\tau_n:=\tau(Z_n)$ for brevity. We first prove that \eqref{1.1} and \eqref{1.2} imply \eqref{1.3}. Observe that each closed set $K\subset\Omega_E$ contains $o(n)$ points of $Z_n$ as $n\to\infty,$ i.e.
\begin{align} \label{5.3}
\lim_{n\to\infty} \tau_n(K) = 0.
\end{align}
This fact follows because $\min_{z\in K} g_E(z,\infty) > 0$ and
\[
0 \le \tau_n(K) \min_{z\in K} g_E(z,\infty) \le  \frac{1}{n} \sum_{z_{k,n}\in K} g_E(z_{k,n},\infty) \le m_E(Z_n) \to 0 \quad\mbox{as } n\to\infty.
\]
Thus if $R>0$ is sufficiently large, so that $E\subset D_R:=\{z:|z|<R\},$ we have $o(n)$ points of $Z_n$ in $\C\setminus D_R.$ Another consequence of the above inequalities is that
\[
\lim_{n\to\infty} \frac{1}{n} \sum_{|z_{k,n}|\ge R} g_E(z_{k,n},\infty) = 0.
\]
Recall that $\lim_{z\to\infty} (g_E(z,\infty)-\log|z|) = V_E,$ see \cite[p. 83]{Ts}. It follows that for any $\eps>0$, there is a sufficiently large $R>0$ such that $V_E-\eps<g_E(z,\infty)-\log|z|<V_E+\eps$ for $|z|\ge R$, and
\[
\frac{o(n)}{n}(V_E-\eps) \le \frac{1}{n} \sum_{|z_{k,n}|\ge R} g_E(z_{k,n},\infty) - \frac{1}{n} \sum_{|z_{k,n}|\ge R} \log|z_{k,n}| \le \frac{o(n)}{n}(V_E+\eps).
\]
Therefore, \eqref{1.3}({\it ii}) is proved by passing to the limit as $n\to\infty$.

Consider
\[
\hat\tau_n := \frac{1}{n} \sum_{|z_{k,n}|<R} \delta_{z_{k,n}}.
\]
Since $\supp(\hat\tau_n)\subset D_R,\ n\in\N,$ we use Helly's theorem \cite[p. 3]{ST} to select a weak-* convergent subsequence from the sequence $\hat\tau_n$. Preserving the same notation for this subsequence, we assume that $\hat\tau_n \stackrel{*}{\rightarrow} \tau$ as $n\to\infty$. It is clear from \eqref{5.3} that $\tau_n \stackrel{*}{\rightarrow} \tau$  as $n\to\infty$, and that $\tau$ is a probability measure supported on the compact set $\hat E := \overline{\C}\setminus\Omega_E.$ Suppose that $R>0$ is large, and order $z_{k,n}$ as follows
\[
|z_{1,n}| \le |z_{2,n}| \le \ldots \le |z_{m_n,n}| < R \le |z_{m_n+1,n}| \le \ldots \le |z_{n,n}|.
\]
Then
\begin{align} \label{5.4}
\hat{I}[\tau_n]  &= \hat{I}[\hat \tau_n] - \frac{2}{n(n-1)} \sum_{1\le j<k \atop m_n<k\le n} \log|z_{j,n}-z_{k,n}| \ge \hat{I}[\hat \tau_n] - \frac{2}{n} \sum_{k=m_n+1}^n \log(2|z_{k,n}|) \\ \nonumber &= \hat{I}[\hat \tau_n] - \frac{2(n-m_n)}{n} \log{2} - \frac{2}{n} \sum_{k=m_n+1}^n \log|z_{k,n}|,
\end{align}
where we used that $|z_{j,n}-z_{k,n}| \le 2\max(|z_{j,n}|,|z_{k,n}|)=2|z_{k,n}|$ for $j<k.$ Note that $\lim_{n\to\infty} m_n/n = 1$ by \eqref{5.3}. For any $\eps>0$, we find $R>0$ such that
\[
\limsup_{n\to\infty} \frac{2}{n} \sum_{k=m_n+1}^n \log|z_{k,n}| = \limsup_{n\to\infty} \frac{2}{n} \sum_{|z_{k,n}|\ge R} \log|z_{k,n}| < \eps
\]
by \eqref{1.3}({\it ii}). Thus we obtain from \eqref{5.4}, \eqref{1.1} and the above estimate that
\begin{align} \label{5.5}
\limsup_{n\to\infty} \hat{I}[\hat \tau_n] \le \limsup_{n\to\infty} \hat{I}[\tau_n] + \limsup_{n\to\infty} \frac{2}{n} \sum_{k=m_n+1}^n \log|z_{k,n}| < V_E + \eps.
\end{align}

We now follow a standard potential theoretic argument to show that $\tau=\mu_E.$ Let $K_M(z,t) := \min\left(-\log{|z-t|},M\right).$ It is clear that $K_M(z,t)$ is a continuous function in $z$ and $t$
on $\C\times\C$, and that $K_M(z,t)$ increases to
$-\log|z-t|$ as $M\to\infty.$ Using the Monotone
Convergence Theorem and the weak-* convergence of $\hat\tau_n\times\hat\tau_n$
to $\tau\times\tau,$ we obtain for the energy of $\tau$ that
\begin{align*}
I[\tau] &=-\iint \log|z-t|\,d\tau(z)\,d\tau(t) =
\lim_{M\to\infty} \left( \lim_{n\to\infty} \iint K_M(z,t)\,
d\hat\tau_n(z)\,d\hat\tau_n(t) \right) \\ &= \lim_{M\to\infty} \left(
\lim_{n\to\infty} \left( \frac{2}{n^2} \sum_{1\le j<k\le m_n} K_M(z_{j,n},z_{k,n}) + \frac{M}{n} \right) \right) \\ &\le
\lim_{M\to\infty} \left( \liminf_{n\to\infty} \frac{2}{n^2}
\sum_{1\le j<k\le m_n} \log\frac{1}{|z_{j,n}-z_{k,n}|} \right) \\ &= \liminf_{n\to\infty} \frac{m_n(m_n-1)}{n^2}
\hat{I}[\hat \tau_n] < V_E+\eps,
\end{align*}
where we applied \eqref{5.5} and $\lim_{n\to\infty} m_n/n = 1$ in the last estimate. Since $\eps>0$ is arbitrary, we conclude that $I[\tau]\le V_E$. Recall that $\supp(\tau) \subset \hat E = \overline{\C}\setminus\Omega_E,$ where $V_{\hat E}=V_E$ and $\mu_{\hat E}=\mu_E$ by \cite[pp. 79-80]{Ts}. Note also that $I[\nu]>V_{\hat E}$ for any probability measure $\nu\neq\mu_{\hat E},\ \supp(\nu)\subset \hat E$, see \cite[pp. 79-80]{Ts}. Hence $\tau=\mu_{\hat E}=\mu_E$ and \eqref{1.3}({\it i}) follows.

Let us turn to the converse statement \eqref{1.3} $\Rightarrow$ \eqref{1.2}. As in the first part of the proof, we note that
$\lim_{z\to\infty} (g_E(z,\infty)-\log|z|) = V_E.$ For any $\eps>0$, we choose $R>0$ so large that $E\subset D_R$ and $|g_E(z,\infty)-\log|z|-V_E|<\eps$ when $|z|\ge R.$ Thus we have from \eqref{1.3}({\it i}) that
\[
\frac{1}{n} \sum_{|z_{k,n}|\ge R} g_E(z_{k,n},\infty) \le \frac{1}{n} \sum_{|z_{k,n}|\ge R} \log|z_{k,n}| + \frac{o(n)}{n}(V_E+\eps).
\]
Increasing $R$ if necessary, we can achieve that
\[
\frac{1}{n} \sum_{|z_{k,n}|\ge R} \log|z_{k,n}| < \eps
\]
for large $n\in\N$ by \eqref{1.3}({\it ii}), which implies that
\begin{align} \label{5.6}
\limsup_{n\to\infty} \frac{1}{n} \sum_{|z_{k,n}|\ge R} g_E(z_{k,n},\infty) \le \eps.
\end{align}
On setting $g_E(z,\infty) = V_E - U^{\mu_E}(z),\ z\in\C,$ we continue $g_E(z,\infty)$ as a subharmonic function in $\C.$ Since $g_E(z,\infty)$ is now upper semi-continuous  in $\C,$ we obtain from \eqref{1.3}({\it i}) and Theorem 0.1.4 of \cite[p. 4]{ST} that
\begin{align} \label{5.7}
\limsup_{n\to\infty} \frac{1}{n} \sum_{|z_{k,n}| < R} g_E(z_{k,n},\infty) &= \limsup_{n\to\infty} \int_{D_R} g_E(z,\infty)\,d\tau_n(z) \le \int_{D_R} g_E(z,\infty)\,d\mu_E(z) \\ \nonumber &= V_E - \int U^{\mu_E}(z)\,d\mu_E(z) = V_E - I[\mu_E] = 0,
\end{align}
where the last equality follows as the energy $I[\mu_E] = V_E,$ see \cite[p. 55]{Ts}. Observe from the definition of $m_E(Z_n)$ and \eqref{5.6}-\eqref{5.7} that
\begin{align*}
0 &\le \limsup_{n\to\infty} m_E(Z_n) \le \limsup_{n\to\infty} \frac{1}{n} \sum_{k=1}^n g_E(z_{k,n},\infty) \le \eps.
\end{align*}
We now let $\eps\to 0,$ to obtain that
\begin{align} \label{5.8}
\lim_{n\to\infty} \frac{1}{n} \sum_{k=1}^n g_E(z_{k,n},\infty) = \lim_{n\to\infty} m_E(Z_n) = 0.
\end{align}
\end{proof}

\begin{remark} \label{rem5.1}
Since \eqref{1.1} and \eqref{1.2} imply \eqref{1.3}, and \eqref{1.3} implies \eqref{5.8} by the above proof, we arrive at
\[
\lim_{n\to\infty} \frac{1}{n} \sum_{k=1}^n g_E(z_{k,n},\infty) = 0.
\]
Hence the sets $Z_n$ satisfying \eqref{1.1} and \eqref{1.2} essentially avoid irregular points of $E$ (in the bulk).
\end{remark}

\begin{proof}[Proof of Theorem \ref{thm1.2}]
Using the definition of $m_E(Z_n)$ and \eqref{5.2}, we obtain that \eqref{1.7} implies \eqref{1.2} because
\begin{align*}
0 \le m_E(Z_n) &\le \frac{1}{n} \sum_{k=1}^n g_E(z_{k,n},\infty) = \int g_E(z,\infty)\,d\tau_n(z) \\ &= \int \left(\int\log|z-t|\,d\mu_E(t) - \log \ca(E) \right)\,d\tau_n(z) \\ &= \int\int\log|z-t|\,d\tau_n(z)\,d\mu_E(t) - \log \ca(E) \\ &= \int\log|P_n(t)|^{1/n}\,d\mu_E(t) - \log \ca(E) \le \log\|P_n\|_E^{1/n} - \log \ca(E).
\end{align*}
Since we assume that \eqref{1.1} holds true, \eqref{1.2} is equivalent to \eqref{1.3} by Theorem \ref{thm1.1}. Thus it remains to show that  \eqref{1.3} implies \eqref{1.7}. For any $\eps>0$, we find $R>0$ such that $E\subset D_R=\{z:|z|<R\}$ and
\[
\lim_{n\to\infty} \left( \prod_{|z_{k,n}|\ge R} |z_{k,n}| \right)^{1/n} < 1 + \eps
\]
by \eqref{1.3}({\it ii}). Since there are $o(n)$ numbers $z_{k,n}$ outside $D_R$ by \eqref{1.3}({\it i}), and since $\|z-z_{k,n}\|_E \le 2|z_{k,n}|$ for $|z_{k,n}|\ge R$, we obtain that
\begin{align*}
\limsup_{n\to\infty} \left\|\prod_{|z_{k,n}| \ge R} (z-z_{k,n})\right\|_E^{1/n} \le \limsup_{n\to\infty}\, 2^{o(n)/n} \left( \prod_{|z_{k,n}|\ge R} |z_{k,n}| \right)^{1/n} \le 1 + \eps.
\end{align*}
Let $\|P_n\|_E=|P_n(z_n)|,\ z_n\in E,$ and assume $\lim_{n\to\infty} z_n = z_0\in E$ by compactness. Define
\[
\hat\tau_n := \frac{1}{n} \sum_{|z_{k,n}|<R} \delta_{z_{k,n}},
\]
and note that $\hat\tau_n \stackrel{*}{\rightarrow} \mu_E$ as $n\to\infty$ by \eqref{1.3}({\it i}). For the polynomial
\[
\hat P_n(z) := \prod_{|z_{k,n}| < R} (z-z_{k,n}),
\]
we have by the Principle of Descent (Theorem I.6.8 of \cite{ST}) that
\begin{align*}
\limsup_{n\to\infty} |\hat P_n(z_n)|^{1/n} = \limsup_{n\to\infty} \exp \left(-U^{\hat\tau_n}(z_n)\right) \le \exp\left(-U^{\mu_E}(z_0)\right) = \ca(E),
\end{align*}
where the last equality is a consequence of Frostman's theorem \eqref{5.1} and the regularity of $E$. It is known that $\|P_n\|_E \ge (\ca(E))^n$, see \cite[p. 16]{AB}. We use this fact together with the above estimates to obtain that
\begin{align*}
 \ca(E) &\le \limsup_{n\to\infty} \|P_n\|_E^{1/n} \le \limsup_{n\to\infty} |\hat P_n(z_n)|^{1/n} \limsup_{n\to\infty} \left(\prod_{|z_{k,n}| \ge R} |z_n-z_{k,n}|\right)^{1/n} \\ &\le (1 + \eps)\,\ca(E).
\end{align*}
Letting $\eps\to 0,$ we obtain \eqref{1.7}.
\end{proof}

\begin{proof}[Proof of Theorem \ref{thm2.1}]
Given $r>0$, define the measures $\nu_k^r$ with $d\nu_k^r(z_{k,n} + re^{it}) = dt/(2\pi),\ t\in[0,2\pi).$ Let $\tau_n:=\tau(Z_n)$ and
\[
\tau_n^r:=\frac{1}{n}\sum_{k=1}^n \nu_k^r,
\]
and estimate
\begin{align} \label{5.9}
\left|\int\phi\,d\tau_n - \int\phi\,d\tau_n^r\right| \le \frac{1}{n}\sum_{k=1}^n \frac{1}{2\pi}\int_0^{2\pi} \left|\phi(z_{k,n}) - \phi(z_{k,n} + re^{it})\right|\,dt \le \omega_{\phi}(r).
\end{align}
We now assume that $E$ is bounded by finitely many piecewise smooth curves, and remove this assumption in the end of proof. Let $g_E(z,\infty) = V_E - U^{\mu_E}(z),\ z\in\C.$ Since $E$ is regular \cite[p. 104]{Ts}, we have that  $g_E(z,\infty) = 0,\ z\in\C\setminus\Omega_E.$ Consider the signed measure $\sigma:=\tau_n^r-\mu_E,\ \sigma(\C)=0.$ This measure is recovered from its potential by the formula
\[
d\sigma=-\frac{1}{2\pi}\left(\frac{\partial U^{\sigma}}{\partial n_+} + \frac{\partial U^{\sigma}}{\partial n_-}\right) ds,
\]
where $ds$ is the arclength on $\supp(\sigma) = \supp(\mu_E) \cup \left( \cup_{k=1}^n \{z:|z-z_{k,n}|=r\}\right)$, and $n_{\pm}$ are the inner and the outer normals. The above representation follows from Theorem 1.1 of \cite{Pr2}, see also Example 1.2 there. Let $D_R:=\{z:|z|<R\}$ be a disk containing the support of $\phi.$ We use Green's identity
\[
\iint_G u \Delta v\,dA =  \int_{\partial G} u\,\frac{\partial v}{\partial n}\,ds - \iint_G \nabla u \cdot \nabla v\,dA
\]
with $u=\phi$ and $v=U^{\sigma}$ in each connected component $G$ of $D_R\setminus\supp(\sigma).$ Since $U^{\sigma}$ is harmonic in $G$, we have that $\Delta U^{\sigma}=0$ in $G$. Adding Green's identities for all domains $G$, we obtain that
\begin{align} \label{5.10}
\left|\int\phi\,d\sigma\right| = \frac{1}{2\pi} \left| \iint_{D_R} \nabla \phi \cdot \nabla U^{\sigma} \,dA \right| \le \frac{1}{2\pi} \sqrt{D[\phi]}\,\sqrt{D[U^{\sigma}]},
\end{align}
by the Cauchy-Schwarz inequality. It is known that $D[U^{\sigma}]=2\pi I[\sigma]$ \cite[Thm 1.20]{La}, where $I[\sigma]=-\iint \log|z-t|\,d\sigma(z)\,d\sigma(t) = \int U^{\sigma}\,d\sigma$ is the energy of $\sigma$. We observe that $\int U^{\mu_E}\,d\mu_E = I[\mu_E] = V_E$, so that
\[
I[\sigma]=\int U^{\tau_n^r}\,d\tau_n^r - 2\int U^{\mu_E}\,d\tau_n^r + V_E.
\]
Since $g_E(z,\infty)$ is harmonic in $\Omega_E$, the mean value property gives that
\begin{align*}
-\int U^{\mu_E}\,d\tau_n^r &= \int \left( g_E(z,\infty) - V_E \right)\, d\tau_n^r(z) \\ &= \frac{1}{n} \left(\sum_{d_E(z_{k,n})\le r} \int g_E(z,\infty)\,d\nu_k^r(z) + \sum_{d_E(z_{k,n})>r} \int g_E(z,\infty)\,d\nu_k^r(z) \right) - V_E \\ &\le \frac{1}{n} \left(\sum_{d_E(z_{k,n})\le r} \max_{d_E(z)\le 2r} g_E(z,\infty) + \sum_{d_E(z_{k,n})>r} g_E(z_{k,n},\infty)\right) - V_E \\ &\le\max_{d_E(z)\le 2r} g_E(z,\infty) + m_E(Z_n) - V_E.
\end{align*}
Taking into account the representation \cite[p. 22]{ST}
\[
U^{\nu_k^r}(z)=-\log\max(r,|z-z_{k,n}|),\quad z\in\C,
\]
we further deduce that
\begin{align*}
\int U^{\tau_n^r}\,d\tau_n^r &= \frac{1}{n^2} \sum_{j,k=1}^n \int U^{\nu_k^r}\,d\nu_j^r \le \frac{1}{n^2} \left(\sum_{j\neq k} \log\frac{1}{|z_{j,n}-z_{k,n}|} - n\log{r}\right) \\ &= \frac{n-1}{n} \hat{I}[\tau_n] - \frac{\log{r}}{n},
\end{align*}
and combine the energy estimates to obtain
\[
I[\sigma] \le 2m_E(Z_n) + \frac{n-1}{n} \hat{I}[\tau_n] - V_E - \frac{\log{r}}{n} + 2\max_{d_E(z)\le 2r} g_E(z,\infty).
\]
Using \eqref{5.9}, \eqref{5.10} and the above estimate, we proceed to \eqref{2.1}-\eqref{2.2} via the following
\begin{align*}
\left|\int\phi\,d\tau_n - \int\phi\,d\mu_E\right| &\le \left|\int\phi\,d\tau_n - \int\phi\,d\tau_n^r\right| + \left|\int\phi\,d\tau_n^r - \int\phi\,d\mu_E\right| \\ &\le \omega_{\phi}(r) + \frac{\sqrt{D[\phi]}\sqrt{D[U^{\sigma}]}}{2\pi} = \omega_{\phi}(r) + \sqrt{\frac{D[\phi]}{2\pi}}\,\sqrt{I[\sigma]}.
\end{align*}
Thus we proved the result for sets bounded by finitely many piecewise smooth curves. To show that \eqref{2.1}-\eqref{2.2} hold for an arbitrary compact set $E$  of positive capacity, we approximate $E$ by a decreasing sequence $E_m,\ m\in\N,$  of compact sets with piecewise smooth boundaries. Let $\eps_1=1$ and consider an open cover of $E$ by the disks $\{D(z,\eps_1)\}_{z\in E},$ where $D(z,\eps_1)$ is centered at $z$ and has radius $\eps_1.$ There exists a finite subcover such that
$E\subset\cup_{k=1}^{N_1} D(c_{k,1},\eps_1).$ Define $E_1:=\cup_{k=1}^{N_1} \overline D(c_{k,1},\eps_1).$ We construct the sets $E_m$ inductively for $m\ge 2.$ Set $\eps_m:={\rm dist}(E,\partial E_{m-1})/2 >0.$ As before, we have a finite subcover such that
\[
E\subset\bigcup_{k=1}^{N_m} D(c_{k,m},\eps_m),\quad m\in\N,
\]
where $c_{k,m}\in E,\ k=1,\ldots,N_m.$ Let
\[
E_m := \bigcup_{k=1}^{N_m} \overline{D(c_{k,m},\eps_m)},\quad m\in\N,
\]
and note that $E_m\subset E_{m-1}$ and $\eps_m\le\eps_{m-1}/2,\ m\ge 2.$ Clearly, the boundary of every $E_m$ consists of finitely many piecewise smooth curves, and each curve is composed of finitely many circular arcs. Thus \eqref{2.1}-\eqref{2.2} hold for every $E_m,\ m\in\N.$ Observe that $\lim_{m\to\infty} \eps_m = 0,$ so that
\[
E=\bigcap_{m=1}^{\infty} E_m.
\]
If $g_{E_m}(z,\infty)$ is the Green function for $\overline\C\setminus E_m$ with pole at $\infty$, then
\[
g_{E_m}(z,\infty)\le g_E(z,\infty),\quad z\in\C,
\]
for any $m\in\N,$ by Corollary 4.4.5 of \cite[p. 108]{Ra}. This gives that
\[
\max_{d_{E_m}(z)\le 2r} g_{E_m}(z,\infty) \le \max_{d_{E_m}(z)\le 2r} g_E(z,\infty),\quad m\in\N.
\]
Since $g_E(z,\infty)$ is subharmonic in $\C$ and harmonic in $\Omega_E$, the maximum on the right of the above inequality is attained on the set $\{z\in\C:d_{E_m}(z)=2r\}\subset\Omega_E.$  We have that
\[
\lim_{m\to\infty} \max_{d_{E_m}(z)=2r} g_E(z,\infty) = \max_{d_E(z)=2r} g_E(z,\infty),
\]
because $d_{E_m}(z) \le d_E(z) \le d_{E_m}(z) + \eps_m,\ z\in\C,$ by the triangle inequality. Thus
\begin{align} \label{5.11}
\limsup_{m\to\infty} \max_{d_{E_m}(z)\le 2r} g_{E_m}(z,\infty) \le \max_{d_E(z)\le 2r} g_E(z,\infty).
\end{align}
Furthermore, Theorem 4.4.6 of \cite[p. 108]{Ra} implies that
\[
\lim_{m\to\infty} g_{E_m}(z,\infty) = g_E(z,\infty),\quad z\in\Omega_E,
\]
so that
\begin{align} \label{5.12}
\lim_{m\to\infty} m_{E_m}(Z_n) = m_E(Z_n).
\end{align}
Recall that $g_{E_m}(z,\infty)=V_{E_m}-U^{\mu_{E_m}}(z),\ z\in\C,$ and the same formula holds with $E_m$ replaced by $E$. Using Theorem 5.1.3 of \cite[p. 128]{Ra}, we obtain that
\begin{align} \label{5.13}
\lim_{m\to\infty} V_{E_m} = V_E,
\end{align}
which gives that
\[
\lim_{m\to\infty} U^{\mu_{E_m}}(z) = U^{\mu_E}(z),\quad z\in\Omega_E.
\]
Since $\supp(\mu_{E_m})\subset\partial\Omega_{E_m}\subset\partial E_m,\ m\in\N,$ we can select a subsequence of measures $\mu_j:=\mu_{E_{m_j}}$ such that $\mu_j \stackrel{*}{\rightarrow} \mu,$ see \cite[p. 3]{ST}. It follows that $\mu$ is a probability measure supported on $\partial\Omega_E,$ as $E=\cap_{m=1}^{\infty} E_m$ and $\overline\Omega_{E_m} \subset \Omega_{E_{m+1}} \subset \Omega_E,\ m\in\N.$ Hence we have by the weak-* convergence that
\[
\lim_{j\to\infty} U^{\mu_j}(z) = U^{\mu}(z),\quad z\in\Omega_E,
\]
which means that
\[
U^{\mu}(z) = U^{\mu_E}(z),\quad z\in\Omega_E.
\]
Since $\supp(\mu)\subset\partial\Omega_E$ and $\supp(\mu_E)\subset\partial\Omega_E$, Carleson's Unicity Theorem \cite[p. 123]{ST} implies that $\mu=\mu_E$. This argument applies to any subsequence of the sequence $\mu_{E_m},\ m\in\N,$ therefore we conclude that $\mu_{E_m} \stackrel{*}{\rightarrow} \mu_E$ as $m\to\infty.$ Consequently,
\[
\lim_{m\to\infty} \int \phi\,d\mu_{E_m} = \int \phi\,d\mu_E.
\]
We now pass to the limit in \eqref{2.1} stated for $E_m$, as $m\to\infty$, and use the above equation together with \eqref{5.11}, \eqref{5.12} and \eqref{5.13} to prove that \eqref{2.1}-\eqref{2.2} also hold for $E$.

\end{proof}

\begin{proof}[Proof of Theorem \ref{thm2.2}]
Since $Z_n\subset E$ and $\lim_{z\to\infty} (g_E(z,\infty)-\log|z|) = V_E,$ the function $\frac{1}{n}\log|P_n(z)|+ V_E - g_E(z,\infty)$ is harmonic in $\Omega_E$ (including $\infty$, where it has value $0$). By the Maximum-Minimum Principle, it is sufficient to prove \eqref{2.5} for $z\in\Gamma_n:=\partial\Omega_n=\{z\in\Omega_E: g_E(z,\infty)=1/n\},$ where $n\ge 2.$ Define the distance between $E$ and $ \Gamma_n$ by
\[
\rho_n:=\dist(E,\Gamma_n)=\min_{t\in E,w\in\Gamma_n} |t-w|,
\]
and note that
\begin{align} \label{5.14}
\rho_n \le |z-t| + d_E(t),\quad t\in\C,\ z\in\Gamma_n,
\end{align}
by the triangle inequality. Let $\diam(E):=\max_{t,w\in E} |t-w|$ be the diameter of $E$, and set $R:=\diam(E)+1.$ We apply Theorem \ref{thm2.1} with the function
\begin{align} \label{5.15}
\phi(t):=\min\left(\log(|z-t|+d_E(t))-\log{R},0\right), \quad t\in\C,\ z\in\Gamma_n.
\end{align}
It is clear that $\supp(\phi)\subset D(z,R):=\{t\in\C:|t-z|<R\}$. Furthermore, $E\subset\supp(\phi)$ for all large $n\in\N,$ because $d_E(z)\to 0$ for $z\in\Gamma_n$ as $n\to\infty$ by the continuity of $g(z,\infty)$ in $\overline\Omega_E.$ One readily finds from the triangle inequality that
\[
\left| |z-t_1| - |z-t_2| \right| \le |t_1-t_2| \quad \forall\ t_1,t_2\in\C,
\]
and
\[
\left| d_E(t_1) - d_E(t_2) \right| \le |t_1-t_2| \quad \forall\ t_1,t_2\in\C,
\]
see also Federer \cite[p. 434]{Fed} for more details about the function $d_E.$ Hence the function $f(t):=|z-t|+d_E(t),\ t\in\C,$ satisfies the Lipschitz condition
\[
\left| f(t_1) - f(t_2) \right| \le 2 |t_1-t_2| \quad \forall\ t_1,t_2\in\C.
\]
Thus the partial derivatives $f_x$ and $f_y$ exist a.e. with respect to the area measure (and the linear measure on vertical and horizontal lines), and we obtain that
\[
|f_x(t)|\le 2 \quad\mbox{and}\quad |f_y(t)|\le 2 \quad \mbox{for a.e. } t=x+iy\in\C.
\]
Hence $\phi_x$ and $\phi_y$ also exist a.e. in the same sense, with
\[
|\phi_x(t)|\le \frac{2}{|z-t|+d_E(t)} \le \frac{2}{\rho_n}
\]
and
\[
|\phi_y(t)|\le \frac{2}{|z-t|+d_E(t)} \le \frac{2}{\rho_n}
\]
for a.e. $t=x+iy\in\C$ by \eqref{5.14}. This gives the estimates
\[
|\phi(t_1)-\phi(t_2)|\le |t_1-t_2|\, \sup_{\C} \sqrt{\phi_x^2+\phi_y^2} \le \frac{2\sqrt{2}}{\rho_n}\, |t_1-t_2|
\]
and
\begin{align} \label{5.16}
\omega_\phi(r) \le \frac{2\sqrt{2}}{\rho_n}\, r.
\end{align}
Furthermore, we obtain for the Dirichlet integral
\begin{align*}
D[\phi] &= \iint_{\C} (\phi_x^2 +\phi_y^2)\,dA \le \iint_{D(z,R)} \frac{8\,dA(t)}{(|z-t|+d_E(t))^2} \\ &\le \iint_{|z-t|\le\rho_n} \frac{8\,dA(t)}{(|z-t|+d_E(t))^2} + \iint_{\rho_n\le |z-t|\le R} \frac{8\,dA(t)}{(|z-t|+d_E(t))^2} \\ &\le 8\left( \frac{1}{\rho_n^2}\,\pi\rho_n^2 + \int_0^{2\pi} \int_{\rho_n}^R \frac{r\,dr}{r^2}\,d\theta \right) = 8\pi\left(1 + 2\log\frac{R}{\rho_n}\right)
\end{align*}
by $\supp(\phi)\subset D(z,R)$ and \eqref{5.14}.
If the Green function $g(z,\infty)$ satisfies the H\"older condition \eqref{2.3}, then $\rho_n\ge (Cn)^{-1/s}$ with $C=C(E)>0$, and $D[\phi]=O(\log{n})$ as $n\to\infty.$ Letting $r=n^{-1/2-1/s},$ we obtain that $\omega_{\phi}(r)=O(n^{-1/2})$ by \eqref{5.16}. Since $2r\le\rho_n$ for large $n$,  we have that
\[
\max_{d_E(z)\le 2r} g_E(z,\infty) \le \frac{1}{n}.
\]
Applying the above estimates and \eqref{2.4} in \eqref{2.1}-\eqref{2.2}, we arrive at
\begin{align} \label{5.17}
\left|\frac{1}{n} \sum_{k=1}^n \phi(z_{k,n}) - \int\phi\,d\mu_E\right| &\le O(n^{-1/2}) + O(\sqrt{\log{n}}) \left( O\left(\frac{\log{n}}{n}\right) + \frac{2}{n} \right)^{1/2} \\ \nonumber &\le O\left(\frac{\log{n}}{\sqrt{n}}\right)\quad \mbox{as }n\to\infty,
\end{align}
where we also used that $m_E(Z_n)=0$. Note that all constants in $O$ terms are independent of the point $z\in\Gamma_n$, of the set $Z_n$, as well as of $n\ge 2.$ It remains to observe that $\phi(t)=\log|z-t| - \log{R}$ for $t\in E$, so that
\begin{align*}
\frac{1}{n} \sum_{k=1}^n \phi(z_{k,n}) - \int\phi\,d\mu_E &= \frac{1}{n}\log|P_n(z)| - \int\log|z-t|\,d\mu_E(t) \\ &= \frac{1}{n}\log|P_n(z)| + V_E - g_E(z,\infty),\quad z\in\Gamma_n,
\end{align*}
by \eqref{5.2}. Thus \eqref{2.5} follows from \eqref{5.17} by the Maximum-Minimum Principle. Further, we obtain from \eqref{2.5} for $z\in\Gamma_n$ that
\begin{align} \label{5.18}
\log \|P_n\|_E \le \log \|P_n\|_{\Gamma_n} \le C_2\,\sqrt{n}\log{n} - nV_E + 1, \quad n\ge 2,
\end{align}
which proves \eqref{2.6}.

For the proof of \eqref{2.7}, we write
\[
P_n'(w) = \frac{1}{2\pi i} \int_{|t-w|=\rho_n} \frac{P_n(t)\,dt}{(t-w)^2}, \quad w\in\partial E.
\]
Hence
\[
|P_n'(w)| \le \frac{\max_{|t-w|=\rho_n} |P_n(t)|}{\rho_n}, \quad w\in\partial E,
\]
and
\[
\|P_n'\|_E \le \rho_n^{-1} \|P_n\|_{\Gamma_n} \le (Cn)^{1/s} \|P_n\|_{\Gamma_n},
\]
because $\rho_n \ge (Cn)^{-1/s}$. Note that
\[
\exp\left( - n(n-1) \hat I[\tau(Z_n)] \right) = \prod_{k=1}^n |P_n'(z_{k,n})| \le (Cn)^{n/s} \|P_n\|_{\Gamma_n}^n.
\]
Thus \eqref{2.7} follows from \eqref{5.18} and the above equation.

\end{proof}

\begin{proof}[Proof of Corollary \ref{cor2.3}]
We first observe that the Fekete points $\cF_n$ satisfy
\begin{align} \label{5.19}
\hat I[\tau(\cF_n)] \le V_E, \quad n\ge 2,
\end{align}
This fact holds because the discrete energies of Fekete sets increase to $V_E$ with $n,$ see \cite[p. 153]{Ra}. Hence \eqref{2.4} holds true and Theorem \ref{2.2} applies to $\cF_n$.

It turns out that \eqref{5.19} is also true for the Leja points $\mathcal L_n=\{\xi_k\}_{k=0}^{n-1},\ n\in\N.$ Consider the corresponding Leja polynomials $L_n(z):=\prod_{k=0}^{n-1} (z-\xi_k),\ n\in\N,$ and recall that \[
\|L_n\|_E = |L_n(\xi_n)| = \prod_{k=0}^{n-1} |\xi_n-\xi_k|
\]
by definition. Hence we have for the Vandermonde determinant
\begin{align*}
|V(\mathcal L_n)| = \prod_{0\le j<k\le n-1} |\xi_j-\xi_k| = \prod_{k=1}^{n-1} |L_k(\xi_k)| = \prod_{k=1}^{n-1} \|L_k(\xi_k)\|_E.
\end{align*}
Since $\|P_k\|_E \ge (\ca(E))^k$ holds for any monic polynomial $P_k,\ \deg(P_k)=k,$ see \cite[p. 16]{AB}, we obtain that
\[
|V(\mathcal L_n)| \ge (\ca(E))^{n(n-1)/2}
\]
and
\[
\hat{I}[\tau(\mathcal L_n)] = - \log |V(\mathcal L_n)|^{\frac{2}{n(n-1)}} \le V_E.
\]

\end{proof}

\begin{proof}[Proof of Theorem \ref{thm2.4}]
Suppose that  $\supp(\phi)\subset\{z:|z|\le R\},$ and that $\phi$ satisfies the Lipschitz condition $|\phi(z)-\phi(t)|\le A|z-t|,\ z,t\in\C.$ It is clear that $\omega_{\phi}(r)\le Ar.$ Also, $|\phi_x|\le A$ and $|\phi_y|\le A$ a.e. in $\C,$ so that $D[\phi]\le 2\pi R^2 A^2.$ If the Green function $g(z,\infty)$ satisfies the H\"older condition \eqref{2.3}, then $\rho_n=\min_{t\in E,w\in\Gamma_n} |t-w| \ge (Cn)^{-1/s}$, with $C=C(E)>0$ and $0<s\le 1$, as in the proof of Theorem \ref{thm2.2}. Letting $r=n^{-2/s},$ we obtain that $\omega_{\phi}(r)\le A\,n^{-2/s}$. Since $2r\le\rho_n$ for large $n$,  we have that
\[
\max_{d_E(z)\le 2r} g_E(z,\infty) \le \frac{1}{n}.
\]
Hence \eqref{2.8} follows from \eqref{2.1}-\eqref{2.2} by combining the above estimates with \eqref{2.4}.
\end{proof}

\begin{proof}[Proof of Corollary \ref{cor2.5}]
We let $\phi(z)=\Re(z^m),\ z\in E,$ and extend this function outside of $E$ to a Lipschitz continuous function with compact support in $\C$. Then \eqref{2.8} holds true for this choice of $\phi$, with $m_E(Z_n)=0$ as $Z_n\subset E$. The same argument applies to $\phi(z)=\Im(z^m),\ z\in E,$ continued appropriately. Combining the estimates obtained from \eqref{2.8}, we arrive at \eqref{2.9}.
\end{proof}

\begin{proof}[Proof of Theorem \ref{thm2.6}]
In the fisrt part, we repeat the proof of Theorem \ref{thm2.4}. Namely, we assume that  $\supp(\phi)\subset\{z:|z|\le R\},$ and that $\phi$ satisfies the Lipschitz condition $|\phi(z)-\phi(t)|\le A|z-t|,\ z,t\in\C.$ It is clear that $\omega_{\phi}(r)\le Ar.$ Also, $|\phi_x|\le A$ and $|\phi_y|\le A$ a.e. in $\C,$ so that $D[\phi]\le 2\pi R^2 A^2.$ If the Green function $g(z,\infty)$ satisfies the H\"older condition \eqref{2.3}, then $\rho_n\ge (Cn)^{-1/s}$, with $C=C(E)>0$ and $0<s\le 1$. Letting $r=n^{-2/s},$ we obtain that $\omega_{\phi}(r)\le A\,n^{-2/s}$. Since $2r\le\rho_n$ for large $n$,  we have that
\[
\max_{d_E(z)\le 2r} g_E(z,\infty) \le \frac{1}{n}.
\]
Noting that $V_E=\log{\ca(E)}=0,$ we seek an upper estimate for $\hat I[\tau(Z_n)]$, in order to apply Theorem \ref{thm2.1}. As $P_n$ has integer coefficients and simple zeros, we obtain that its discriminant $\Delta(P_n)=a_n^{2n-2} \left(V(Z_n)\right)^2$ is a non-zero integer \cite[p. 24]{Pra}. Hence $|\Delta(P_n)|\ge 1$ and
\[
\hat I[\tau(Z_n)] = - \frac{1}{n(n-1)} \log|\Delta(P_n)| + \frac{2}{n} \log|a_n| \le \frac{2}{n} \log|a_n|.
\]
Combining the above estimates in \eqref{2.1}-\eqref{2.2}, we obtain \eqref{2.10}.
\end{proof}

\begin{proof}[Proof of Corollary \ref{cor2.7}]
Since the roots $Z_n=\{z_{k,n}\}_{k=1}^n$ of $P_n$ come in complex conjugate pairs, we have that
\[
\frac{1}{n}\sum_{k=1}^n z_{k,n} = \frac{1}{n}\sum_{k=1}^n \Re(z_{k,n}).
\]
Hence we consider $\phi(z)=\Re(z),\ z\in E,$ and extend this function outside of $E$ to a Lipschitz continuous function with compact support in $\C$. Note that $m_E(Z_n)=0$ as $Z_n\subset E$. Also, $|a_n|\le M$. The electrostatic centroid of $E$ lies at the origin because of the set symmetry and uniqueness of the equilibrium measure $\mu_E$, i.e., $\int z\,d\mu_E(z) = 0$. Thus \eqref{2.11} follows from \eqref{2.10} by Theorem \ref{thm2.6}.
\end{proof}

\begin{proof}[Proof of Example \ref{ex2.8}]
We recall some basic facts from prime number theory, which may be found, for example, in Ingham \cite{Ing} and Davenport \cite{Dav}. Let $\pi(x)$ be the number of primes not exceeding $x$. The Prime Number Theorem states that
\[
\pi(x) = \frac{x}{\log{x}} + o\left(\frac{x}{\log{x}}\right) \quad \mbox{as } x\to \infty.
\]
It is equivalent to the following asymptotic formulas from prime number theory. If $p_m$ is the $m$th prime number, then \cite[p. 36]{Ing}
\begin{align} \label{5.20}
p_m = m\log{m} +o(m\log{m}) \quad \mbox{as } m\to\infty.
\end{align}
For the Chebyshev $\theta$-function, we have \cite[p. 13]{Ing}
\begin{align} \label{5.21}
\theta(x) := \sum_{p\le x} \log{p} = x + o(x) \quad \mbox{as } x\to\infty,
\end{align}
where the sum extends over all primes $p\le x.$

An asymptotic of the sum of consecutive primes may be found by using the method described in Rosser and Schoenfeld \cite[pp. 67-68]{RS}, which was carried out by Sal\'at and Zn\'am \cite{SZ}. We use a more precise form of the Prime Number Theorem proved by de la Vall\'ee Poussin, see \cite[p. 65]{Ing} and \cite[p. 113]{Dav}:
\begin{align} \label{5.22}
\pi(x) = \li\,x + O\left(xe^{-a\sqrt{\log{x}}}\right) \quad \mbox{as } x\to\infty,
\end{align}
where $a>0$ and
\[
\li\,x:=\int_2^x \frac{dt}{\log{t}}.
\]
Better error terms are certainly available now, but this classical result is sufficient for our purpose. Suppose that $f:[0,\infty)\to\R$ is a continuously differentiable function. Using the Stieltjes integral and integration by parts, we obtain by following \cite[p. 67]{RS} that
\begin{align*}
\sum_{p\le x} f(p) &= \int_2^x f(t)\,d\pi(t) = f(x)\pi(x) - \int_2^x f'(t)\pi(t)\,dt \\ &= \int_2^x \frac{f(t)\,dt}{\log{t}} + f(x)(\pi(x)-\li\,x) - \int_2^x f'(t)(\pi(t)-\li\,t)\,dt.
\end{align*}
If we set $f(x)=x$ and use \eqref{5.22} in the above formula, it gives that
\begin{align*}
\sum_{p\le x} p = \frac{x^2}{2\log{x}} + o\left(\frac{x^2}{\log{x}}\right) \quad \mbox{as } x\to \infty.
\end{align*}
Taking into account \eqref{5.20}, we arrive at the asymptotic
\begin{align} \label{5.23}
n = \sum_{m=1}^k p_m - k = \frac{k^2 \log k}{2} + o(k^2 \log k)\quad \mbox{as } k\to\infty.
\end{align}
Combining \eqref{5.20}, \eqref{5.21} and \eqref{5.23}, we obtain that
\begin{align*}
\log\left(\prod_{m=1}^k p_m \right) = \theta(p_k) \ge c_1 \sqrt{n\log n},\quad n\ge 2,
\end{align*}
with a constant $c_1>0.$ This implies our estimate for $\|P_n\|_D,$ because \[
|P_n(z)| = \left|\prod_{m=1}^k \left(\sum_{j=0}^{p_m-1} z^j\right) \right| \le \prod_{m=1}^k p_m = P_n(1), \quad |z|\le 1.
\]
Also, \eqref{5.23} immediately gives that
\[
\frac{k}{n} \ge \frac{c_2}{\sqrt{n\log{n}}},\quad n\ge 2,
\]
where $c_2>0.$
\end{proof}

\end{document}